\documentclass[10.5pt]{article}
\usepackage{amsmath}
\usepackage{amsfonts}
\usepackage{graphicx}
\usepackage{booktabs}
\usepackage{amsmath}%\numberwithin{equation}{section}
\usepackage{amssymb}
\usepackage{latexsym}
\usepackage{amsmath,amsfonts,amssymb,amsthm,euscript,makeidx,color,mathrsfs}
 \usepackage{float}

\newtheorem{problem}{Problem}
\newtheorem{theorem}{Theorem}%[section]
\newtheorem{assumption}{Assumption}%[section]
\newtheorem{lemma}{Lemma}%[section]
\newtheorem{remark}{Remark}%[section]
\newtheorem{definition}{Definition}%[section]
\def\rf{\eqref}
\def\no{\nonumber}

\makeatletter
   
   \@addtoreset{equation}{section}
\makeatother

\title{\bf Reinforcement Learning for Stochastic LQ Control of Discrete-Time Systems with Multiplicative Noises
}

\author{ Hongdan Li,\ Lucky Qiaofeng Li, Xun Li and \ Zhaorong Zhang
\thanks{H. Li is with College of Electrical Engineering and Automation, Shandong University of Science and Technology, Qingdao, Shandong, P.R.China 266590. L. Li  is with  College of Computing, Data Science, and Society, University of California, Berkeley, CA 94720, USA. X. Li is with Department of Applied Mathematics, The Hong Kong Polytechnic University, Kowloon, Hong Kong, China.
Z. Zhang is with School of Computer Science and Technology, Shandong University, Qingdao 250100, China.
Z. Zhang is the corresponding author.(e-mail: zhaorong.zhang@uon.edu.au) }% <-this % stops a space
}

\begin{document}

\pagenumbering{arabic}
 \setcounter{page}{1}

\pagenumbering{arabic} \thispagestyle{empty} \setcounter{page}{1}

%{\footnotesize \noindent
\baselineskip 16pt
\date{}
  \maketitle

{\bf Abstract}
This paper considers a stochastic linear quadratic problem for discrete-time systems with multiplicative noises over an infinite horizon. To obtain the optimal solution, we propose an online iterative algorithm of reinforcement learning based on Bellman dynamic programming principle. The algorithm avoids the direct calculation of algebra Riccati equations. It merely takes advantage of state trajectories over a short interval instead of all iterations, significantly simplifying the calculation process. Under the stabilizable initial values, numerical examples shed light on our theoretical results.


\section{Introduction}
In recent years, artificial intelligence and machine learning have been prosperously developed and attached great importance to in commercial and research fields. For example, Lin et al. \cite{lin} investigated machine-learning-based methods for bankruptcy prediction problems; Maschler et al. \cite{bm} studied a type of deep learning algorithms for industrial automation systems; by employing deep neural networks, Giusti et al. \cite{ag} designed an image classification approach for perceiving forest or mountain trails. Among them, reinforcement learning (RL) is widely acknowledged as a powerful tool for tackling problems where partial information is unavailable. Generally, RL methods involve a group of agents that interact with the environment and aim to obtain the maximum rewards by adjusting their actions according to the stimuli received from the environment. More specifically, RL is a class of machine learning techniques where adaptive controllers are manipulated to solve optimal control problems in real-time.

Linear quadratic (LQ) optimal control, including LQ regulation and LQ tracking problems, is a fundamental and critical problem arising in scientific and engineering disciplines and has been extensively studied over the past few decades. LQ optimal control can be effectively addressed by model-based methods, including semi-definite programming, eigenvalue decomposition, and iterative methods, where solving the algebraic Riccati equation (ARE) is inevitable. However, to solve AREs, complete system information should be previously known, which is usually unrealistic in practical applications. Therefore, optimal control problems with incomplete system information are still treated as a major obstacle. In this scenario, RL methods, which can be applied to obtain optimal control policies by learning online the solution to the Hamilton-Jacobi-Bellman (HJB) equation, are playing an increasingly important role in solving model-free optimal control in both continuous-time and discrete-time models. In early times, Werbos \cite{werbos1} put forward the so-called adaptive dynamic programming (ADP) methods, which approximate optimal feedback control policies for discrete-time frameworks using system trajectories. Subsequently, RL methods integrated with ADP have become a research emphasis for designing optimal controllers in \cite{werbos2}. For example, Bertsekas and Tsitsiklis \cite{ber} solved discrete-time optimal control by adopting an RL approach named neuro-dynamic programming dependent on offline solution. For continuous-time (CT) cases, Modares and Lewis \cite{lewis1} designed an RL method, which is effective for solving optimal LQ control without the knowledge of system state matrix $A$, to devise the adaptive controllers with actor-critic structure. For constrained nonlinear systems with saturating actuators, Abu-Khalaf et al. \cite{lewis2} proposed a method based on RL to solve the HJB equation and approximate the optimal constrained input state feedback controller. Modares et al. \cite{lewis3} focused on H$_{\infty}$ tracking problems and designed an online off-policy RL method, which is intended to learn the solution to the HJB equation. More recently, Zhang et al. \cite{kunz} developed a novel approach combining RL and decentralized control design to address interconnected systems' tracking control.

Among RL methods, Q-learning is a type of model-free approach that is guaranteed to converge to the optimal solution when implemented in the environment of Markov decision processes. Established by Watkins and Dayan \cite{watkins}, the celebrated Q-learning has been successfully engaged to solve discrete-time optimal control for linear systems. Based on the Q-learning algorithm, \cite{Frank} studied a linear quadratic tracker for unknown discrete-time systems over an infinite horizon. Lee and Hu \cite{pd} converted the linear quadratic regulator problem into a non-convex optimization problem, which can be solved by applying the Q-learning method with primal-dual update procedures. For LQ output regulation problems, Rizvi and Lin \cite{ay} adopted a Q-learning approach that only requires input-output data rather than full-state data and performs both policy and value iteration. In \cite{jina}, an off-policy Q-learning algorithm was proposed for determining the optimal control policy for affine nonlinear control problems. In \cite{h}, state-data-driven and output-data-driven reinforcement Q-learning algorithms were devised to tackle H$_\infty$ tracking problems in discrete-time linear systems. In \cite{co}, Vamvoudakis and Hespanha proved that the graphical Nash equilibrium of a multi-agent system is guaranteed under a novel cooperative Q-learning algorithm where each Q-function is related to the dynamics of neighbors. Q-learning has also found wide applications in solving LQ stochastic systems with multiplicative noises. Inspired by Q-learning, Du et al. \cite{du} presented an online algorithm that overcomes the challenge of discrete-time LQ optimal control where the dynamics and criteria are both associated with Gaussian noises with inadequate statistical information. Zhang et al. \cite{juan} investigated the non-zero-sum difference game where the statistical data of multiplicative noise is unknown and designed a Q-learning algorithm for deriving the Nash strategy in the finite horizon. However, in the aforementioned literature, acquiring the solution to the associated Riccati equation is the premise of solving optimal policies. In other words, the complete system characteristics should be employed in designing the controllers, which distinctly increases computation complexity. It is notable that \cite{li} recently solved a CT LQ stochastic problem by an RL algorithm, which directly yields the optimal control using only local trajectory information. Nevertheless, CT models are inadaptable for digital signal processing, which goes against the tendency of scientific and technological developments. In this paper, we gave insight into the discrete-time stochastic LQ control problem and proposed an online RL method that merely takes advantage of state trajectories over a short time interval in each iteration and directly approaches the optimal control policy without modeling the inner structure of the system.

The main contributions of this paper can be summarized as follows:

(1) The policy evaluation procedure of the proposed algorithm does not involve solving the related Riccati equation; Bellman dynamic programming is employed instead, which significantly reduces the computation complexity.

(2) The proposed algorithm is able to solve discrete-time stochastic systems where both inputs and states are associated with multiplicative noises.

(3) Given a stabilizable initial controller, we prove that the control policies updated by the online algorithm are all stabilizable without requiring system identification procedures.

(4) During the policy evaluation procedure, the proposed algorithm only requires the trajectory information within an arbitrary time interval.

The rest of this paper is organized as follows. In Section 2, we set up the problem. In Section 3, the main results are presented. In Section 4, the online implementation of the proposed algorithm is discussed. Section 5 extends the algorithm to stochastic systems with partially unknown information. In Section 6, two numerical examples are provided. And Section 7 sums up this paper.

Notation: Define $vec(M)$ as a vectorization map from a matrix $M$ into an $nm$-dimensional column vector for compact representations, which stacks the columns of $M$ on top of one another. $A\otimes B$ denotes a Kronecker product of matrices $A$ and $B$, and $vec(ABC)=(C'\otimes A)vec(B)$. Denote an operator $vec^{+}(P)$, which
maps $P$ into an $N$-dimensional vector by stacking the columns
corresponding to the diagonal and lower triangular parts of $P$ on
top of one another where the off-diagonal terms of $P$ are doubled.

%\vspace{-20pt}
%\begin{center}
\section{Problem  Statement}

\normalsize
Let $x_t\in \mathbb{R}^{n}$ represent the state process with a deterministic initial value $\xi$, $u_t\in \mathbb{R}^{m}$ denote a control process and $w_t$ be independent Gaussian noises with zero mean value and covariance $\sigma$, we consider the following discrete-time stochastic linear system
\begin{eqnarray}
x_{t+1}=Ax_t+Bu_t+(Cx_t+Du_t)w_t,\label{2.1}
\end{eqnarray}
where the coefficient matrices $A, C\in \mathbb{R}^{n\times n}$, $B, D\in \mathbb{R}^{n\times m}$ are all constant matrices.  For simplicity, we denote system \rf{2.1} as $[A, C; B, D]$.

Firstly, we introduce the following definitions.
\begin{definition}
The following autonomous system $[A, C]$
\begin{eqnarray}
x_{t+1}=(A+w_tC)x_t \label{2.2}
\end{eqnarray}
 is called asymptotically mean-square stable if for any initial value $\xi$, there holds
 \begin{eqnarray*}
\lim_{t\rightarrow\infty}\mathbb{E}[x'_tx_t]=0.
\end{eqnarray*}
\end{definition}
\begin{definition}
System $[A, C; B, D]$ is called asymptotically mean-square stabilizable if there exists a constant matrix $K\in\mathbb{R}^{m\times n}$ such that the closed-loop system of \rf{2.1}, i.e.,
\begin{eqnarray}\label{kk}
x_{t+1}=[(A+BK)+(C+DK)w_t]x_{t}
\end{eqnarray}
 is  asymptotically mean-square stable. In this case, $K$ is called a stabilizer of system $[A, C; B, D]$ and feedback control $u(\cdot)=Kx(\cdot)$ is called stabilizing. The set of all stabilizers is denoted by $\mathcal{S}([A, C; B, D])$.
\end{definition}
It is well known that stability is the premise to ensure the normal operation of a system. For discussing, we make the following assumption.
\begin{assumption}\label{a1}
System $[A, C; B, D]$ is asymptotically mean-square stabilizable, i.e., $\mathcal{S}([A, C; B, D])\neq \emptyset$.
\end{assumption}
Under Assumption \ref{a1}, we define the corresponding set of admissible controls as
\begin{eqnarray*}
\mathcal{U}_{ad}=\{u(\cdot)\in L^{2}_{\mathcal{F}}(\mathbb{R}^{m}): u(\cdot) \ \ is \ \ stablilizing\}.
\end{eqnarray*}
\begin{lemma}\label{l1}
A matrix $K\in\mathbb{R}^{m\times n}$ is a stabilizer of system $[A, C; B, D]$ if and only if there exists a matrix $P>0$ such that
\begin{eqnarray*}
(A+BK)'P(A+BK)+\sigma^{2}(C+DK)'P(C+DK)-P<0.
\end{eqnarray*}
In this case, for any symmetric matrix $M (or \geq 0, >0)$, the Lyapunov equation
\begin{eqnarray*}
P=(A+BK)'P(A+BK)+\sigma^{2}(C+DK)'P(C+DK)+M
\end{eqnarray*}
admits a unique solution $P (or \geq 0, >0)$.
\end{lemma}
In this paper, we consider the following quadratic cost functional
\begin{eqnarray}
 J(s,x;u(\cdot))=\mathbb{E}_{s}\Big\{\sum^{\infty}_{t=s}[x'_tQx_t
 +2u'_tSx_t+u'_tRu_t]\Big\}.\label{2.3}
\end{eqnarray}

The weighting matrices satisfy the following standard assumption.
\begin{assumption}\label{a2}
$R\in\mathbb{R}^{m\times m}$ is a positive definite matrix, and $Q-S'R^{-1}S$ is a positive definite matrix.
\end{assumption}
The main purpose of this paper is to solve the following problem.
\begin{problem}\label{p1}
Given $s\geq 0$ and $x$, find a control $u^{\ast}(\cdot)\in\mathcal{U}_{ad}$ such that
\begin{eqnarray*}
 J(s,x;u^{\ast}(\cdot))=\inf_{u(\cdot)\in\mathcal{U}_{ad}}J(s,x;u(\cdot))\triangleq V(s,x),
\end{eqnarray*}
where $V(s,x)$ is called the value function.
\end{problem}
 \begin{remark}\label{r1}
 Under Assumption \ref{a2}, the well-posedness of Problem \ref{p1} can be guaranteed. \end{remark}

 \section{Reinforcement Learning for the Stochastic LQ Problem}
In this section, we present the main results of the convergence of the proposed RL algorithm.
\begin{lemma}\label{l2}
Suppose that $P>0$ is a solution to the following Lyapunov equation
\begin{align}
(A+BK)'P(A+BK)+\sigma^{2}(C+DK)'P(C+DK)+K'RK+K'S+S'K+Q-P=0, \label{2.4}
\end{align}
where
\begin{align}
K=-(R+B'PB+\sigma^{2}D'PD)^{-1}(B'PA+\sigma^{2}D'PC+S),\label{2.04}
\end{align}
then $u=Kx$ is an optimal control of Problem \ref{p1} and $V(s,\xi)=\xi'P\xi$. Moreover, we have the Bellman's DP recursive equation
\begin{align}
\xi'P\xi=\mathbb{E}_{s}\Big\{\sum^{s+l}_{t=s}x'_t(Q+K'S+S'K+K'RK)x_t
 +x'_{s+l+1}Px_{s+l+1}\Big\}\label{2.5}
\end{align}
for any constant l.
\end{lemma}

\begin{remark}\label{r2}
From Lemma \ref{l2}, when initial value x is given, the value of P can be inferred from \rf{2.5} by local trajectories of $x(\cdot)$ on $[s, s+l+1]$ instead of the entire state trajectories of $x(\cdot)$ on $[s,\infty)$.
\end{remark}
At each iteration $i (i=1,2,\ldots)$, the state trajectory is denoted by $x^{(i)}$ corresponding to the control law $K^{(i)}$. Now, we present Algorithm 1 as follows.

\begin{table}[htbp]

   \centering

   \begin{tabular}{l}

      \toprule

     \textbf{Algorithm 1} \emph{Policy Iteration for Problem \ref{p1}}\\

      \midrule

      1:\textbf{Initialization:} Select any stabilizer $K^{(0)}$ for system \rf{2.1}. \\

      2: Let $i=0$ and $\delta>0$.\\

      3: \textbf{do} \{ \\

      4: Obtain local state trajectories $x^{(i)}$ by running system \rf{2.1} with
$K^{(i)}$ on $[s, s+l+1]$. \\

5: \textbf{Policy Evaluation} (Reinforcement): Solve $P^{(i+1)}$ from the identity\\
%\end{tabular}
%\end{table}
%%
%\begin{eqnarray}
\(\xi'P^{(i+1)}\xi-\mathbb{E}_{s}[(x^{(i)}_{s+l+1})'P^{(i+1)}x^{(i)}_{s+l+1}]\)\\
\(=\mathbb{E}_{s}\Big\{\sum^{s+l}_{t=s}(x^{(i)}_t)'[Q+(K^{(i)})'S+S'K^{(i)}
+(K^{(i)})'RK^{(i)}]x^{(i)}_t\Big\}\) \quad ($\ast$)\\
%\end{eqnarray}
%%
%\begin{table}[htbp]
%
%   \centering
%
%   \begin{tabular}{l}
%   \\
   6: \textbf{Policy Improvement} (Update): Update $K^{(i+1)}$ by the formula\\
\(K^{(i+1)}=-(R+B'P^{(i+1)}B+\sigma^{2}D'P^{(i+1)}D)^{-1}
(B'P^{(i+1)}A+\sigma^{2}D'P^{(i+1)}C+S)\)\quad ($\ast\ast$)
\\

7: $i\leftarrow i+1$\\

8:\}\ \textbf{until} $\|P^{(i+1)}-P^{(i)}\|<\delta$.\\

      \bottomrule

   \end{tabular}

\end{table}
\begin{lemma}\label{l3}
Suppose that Assumption \ref{a2} holds and system \rf{2.1} is stabilizable with $K^{(i)}$. Then solving Policy Evaluation ($\ast$) in Algorithm 1 is equivalent to solving Lyapunov Recursion
\begin{eqnarray}
&&(A+BK^{(i)})'P^{(i+1)}(A+BK^{(i)})
+\sigma^{2}(C+DK^{(i)})'P^{(i+1)}(C+DK^{(i)})\no\\
&&+(K^{(i)})'RK^{(i)}+(K^{(i)})'S+S'K^{(i)}+Q-P^{(i+1)}=0. \label{2.8}
\end{eqnarray}
\end{lemma}

\begin{proof}
Suppose that $K^{(i)}$ is a stabilizer for system \rf{2.1}. By Assumption \ref{a2}, we have
\begin{eqnarray*}
(K^{(i)})'RK^{(i)}+(K^{(i)})'S+S'K^{(i)}+Q=Q-S'R^{-1}S+(RK^{(i)}+S)'R^{-1}(RK^{(i)}+S)>0.
\end{eqnarray*}
By Lemma \ref{l1}, Lyapunov equation \rf{2.8} admits a unique solution $P^{(i+1)}>0$.

Let $\mathcal{L}(t)=(x^{(i)}_t)'P^{(i+1)}x^{(i)}_t$. Thus, we have
\begin{eqnarray}
\mathcal{L}(t)-\mathcal{L}(t+1)
&=&(x^{(i)}_t)'[P^{(i+1)}-
(A_{t}+B_{t}K^{(i)})'P^{(i+1)}(A_{t}+B_{t}K^{(i)})]x^{(i)}_t\label{2.9}
\end{eqnarray}
with $A_{t}=A+w_tC, B_{t}=B+w_tD$.

Now, by taking a summation from $t=s$ to $s+l$ and conditional expectation $\mathbb{E}_{s}$ on both sides of \rf{2.9}, we have
\begin{eqnarray}
&&\xi'P^{(i+1)}\xi-\mathbb{E}_{s}[(x^{(i)}_{s+l+1})'P^{(i+1)}x^{(i)}_{s+l+1}]\no\\
&=&\mathbb{E}_{s}\Big\{\sum^{s+l}_{t=s}(x^{(i)}_t)'[P^{(i+1)}-(A+BK^{(i)})'P^{(i+1)}
(A+BK^{(i)})\no\\
&&-\sigma^{2}(C+DK^{(i)})'P^{(i+1)}(C+DK^{(i)})]x^{(i)}_t\Big\}\label{2.10}\\
&=&\mathbb{E}_{s}\Big\{\sum^{s+l}_{t=s}(x^{(i)}_t)'[(K^{(i)})'RK^{(i)}
+(K^{(i)})'S+S'K^{(i)}+Q]x^{(i)}_t\Big\}\no,
\end{eqnarray}
which confirms Policy Evaluation ($\ast$).

On the other hand, if $P^{(i+1)}$ is the solution to ($\ast$), from \rf{2.10}, for any $n>s$, we get
\begin{eqnarray}
0&=&\mathbb{E}_{n}\Big\{\sum^{n+l}_{t=n}(x^{(i)}_t)'[P^{(i+1)}-(A+BK^{(i)})'P^{(i+1)}
(A+BK^{(i)})\no\\
&&-\sigma^{2}(C+DK^{(i)})'P^{(i+1)}(C+DK^{(i)})-(K^{(i)})'RK^{(i)}
-(K^{(i)})'S-S'K^{(i)}-Q]x^{(i)}_t\Big\}.\label{2.100}
\end{eqnarray}
Divide $l$ on both sides of \rf{2.100}, and let $l\rightarrow 0$, we have
\begin{eqnarray}
0&=&(x^{(i)}_n)'[P^{(i+1)}-(A+BK^{(i)})'P^{(i+1)}
(A+BK^{(i)})\no\\
&&-\sigma^{2}(C+DK^{(i)})'P^{(i+1)}(C+DK^{(i)})-(K^{(i)})'RK^{(i)}
-(K^{(i)})'S-S'K^{(i)}-Q]x^{(i)}_n.\label{2.11}
\end{eqnarray}
It implies that \rf{2.8} holds with the arbitrariness of $x^{(i)}(n)$. Moreover, from Lemma \ref{l2}, \rf{2.8} admits a positive definite solution.
\end{proof}

\begin{remark}\label{r3}
Although solving Policy Evaluation ($\ast$) in Algorithm 1 is equivalent to solving Lyapunov Recursion \rf{2.8}, it is necessary to know $A$ for solving \rf{2.8}, i.e., the complete information of the system is required. By comparison, Algorithm 1 can be run with partial information.
\end{remark}

In fact, the precondition for the normal operation of Algorithm 1 is that
$K^{(i)}$ should be stepwise stable. Next, we consider the stepwise stable property of $K^{(i)}$.

\begin{theorem}\label{t1}
Suppose that Assumptions \ref{a1} and \ref{a2} hold, and initial value $K^{(0)}$ is a stabilizer for system \rf{2.1}. Then all the policies $\{K^{(i)}\}^{\infty}_{i=1}$ updated by Policy Improvement ($\ast\ast$) are stabilizers. Moreover, ($\ast$) admits a unique solution $P^{(i+1)}>0$ at each step.
\end{theorem}

\begin{proof}
From Lemma \ref{l3}, the conclusion holds for $i=0$. Now we assume that for $i\geq 1$,
$K^{(i-1)}$ is a stabilizer and  $P^{(i)}>0$ is the unique solution to ($\ast$). Accordingly, we illustrate that $K^{(i)}=-(R+B'P^{(i)}B+\sigma^{2}D'PP^{(i)}D)^{-1}
(B'P^{(i)}A+\sigma^{2}D'P^{(i)}C+S)$ is a stabilizer and $P^{(i+1)}>0$ is the unique solution to ($\ast$). Based on the previous analysis, we have
\begin{eqnarray}
&&(A+BK^{(i)})'P^{(i)}(A+BK^{(i)})
+\sigma^{2}(C+DK^{(i)})'P^{(i)}(C+DK^{(i)})-P^{(i)}\no\\
&=&(A+BK^{(i-1)})'P^{(i)}(A+BK^{(i-1)})
+\sigma^{2}(C+DK^{(i-1)})'P^{(i)}(C+DK^{(i-1)})-P^{(i)}\no\\
&&-\Big[(A+BK^{(i-1)})'P^{(i)}(A+BK^{(i-1)})
+\sigma^{2}(C+DK^{(i-1)})'P^{(i)}(C+DK^{(i-1)})\no\\
&&-(A+BK^{(i)})'P^{(i)}(A+BK^{(i)})
-\sigma^{2}(C+DK^{(i)})'P^{(i)}(C+DK^{(i)})\Big]\no\\
&=&-(K^{(i-1)})'RK^{(i-1)}-(K^{(i-1)})'S-S'K^{(i-1)}-Q\no\\
&&-\Big[(A'P^{(i)}B+\sigma^{2}C'P^{(i)}D)(K^{(i-1)}-K^{(i)})
+(K^{(i-1)}-K^{(i)})'(A'P^{(i)}B+\sigma^{2}C'P^{(i)}D)'\no\\
&&+(K^{(i-1)})'(B'P^{(i)}B+\sigma^{2}D'P^{(i)}D)K^{(i-1)}
-(K^{(i)})'(B'P^{(i)}B+\sigma^{2}D'P^{(i)}D)K^{(i)}\Big]\no\\
&=&-(K^{(i-1)})'RK^{(i-1)}-(K^{(i-1)})'S-S'K^{(i-1)}-Q\no\\
&&+[(R+B'P^{(i)}B+\sigma^{2}D'P^{(i)}D)K^{(i)}+S]'(K^{(i-1)}-K^{(i)})\no\\
&&+(K^{(i-1)}-K^{(i)})'[(R+B'P^{(i)}B+\sigma^{2}D'P^{(i)}D)K^{(i)}+S]\no\\
&&-(K^{(i-1)})'(B'P^{(i)}B+\sigma^{2}D'P^{(i)}D)K^{(i-1)}
+(K^{(i)})'(B'P^{(i)}B+\sigma^{2}D'P^{(i)}D)K^{(i)}\no\\
&=&-(K^{(i-1)})'RK^{(i-1)}-(K^{(i-1)})'S-S'K^{(i-1)}-Q\no\\
&&+(RK^{(i)}+S)'(K^{(i-1)}-K^{(i)})+(K^{(i-1)}-K^{(i)})'(RK^{(i)}+S)\no\\
&&-(K^{(i-1)}-K^{(i)})'(B'P^{(i)}B+\sigma^{2}D'P^{(i)}D)(K^{(i-1)}-K^{(i)})\no\\
&=&-(Q-S'R^{-1}S)-(K^{(i)}+R^{-1}S)'R(K^{(i)}+R^{-1}S)\no\\
&&-(K^{(i-1)}-K^{(i)})'(R+B'P^{(i)}B+\sigma^{2}D'P^{(i)}D)(K^{(i-1)}-K^{(i)})<0,
\label{2.11}
\end{eqnarray}
where the second equality holds for \rf{2.8} in Lemma \ref{l3}, the third equality holds for ($\ast\ast$) and the last inequality results from Assumption \ref{a2}. Hence, from Lemma \ref{l1}, $K^{(i)}$ is a stabilizer. Moreover, by Lemmas \ref{l1} and \ref{l3}, $P^{(i)}>0$ is the unique solution to ($\ast$).

\end{proof}

Next, we are ready to prove the convergence of Algorithm 1.

\begin{theorem}\label{t2}
The sequence of $\{P^{(i)}\}^{\infty}_{i=1}$ updated by Algorithm 1 converges to $P$, which is the solution to the following SARE
\begin{eqnarray}
P&=&Q+A'PA+\sigma^{2}C'PC-(B'PA+\sigma^{2}D'PC+S)'\no\\
&&\times
(R+B'PB+\sigma^{2}D'PD)^{-1}(B'PA+\sigma^{2}D'PC+S).\label{2.13}
\end{eqnarray}
In this case, the optimal control of Problem \ref{p1} is
\begin{eqnarray}
u^{\ast}=Kx^{\ast},\label{2.14}
\end{eqnarray}
where $K$ is expressed as in \rf{2.04}. Moreover, $K\in \mathcal{S}([A, C; B, D])$.
\end{theorem}

\begin{proof}
From Lemma \ref{l3}, we now prove that $\{P^{(i)}\}^{\infty}_{i=1}$
in \rf{2.8} combining with ($\ast\ast$) converges to $P$, which is the solution to SARE \rf{2.13}.

Step 1: The convergence of $\{P^{(i)}\}^{\infty}_{i=1}$.

Denote $\hat{P}^{(i+1)}=P^{(i)}-P^{(i+1)}$, and $\hat{K}^{(i)}=K^{(i-1)}-K^{(i)}$ for $i=1,2,\ldots$, then
\begin{eqnarray}
0&=&(A+BK^{(i-1)})'P^{(i)}(A+BK^{(i-1)})
+\sigma^{2}(C+DK^{(i-1)})'P^{(i)}(C+DK^{(i-1)})\no\\
&&+(K^{(i-1)})'RK^{(i-1)}+(K^{(i-1)})'S+S'K^{(i-1)}+Q-P^{(i)}\no\\
&&-(A+BK^{(i)})'P^{(i+1)}(A+BK^{(i)})
-\sigma^{2}(C+DK^{(i)})'P^{(i+1)}(C+DK^{(i)})\no\\
&&-(K^{(i)})'RK^{(i)}-(K^{(i)})'S-S'K^{(i)}-Q+P^{(i+1)}\no\\
&=&(A+BK^{(i)})'\hat{P}^{(i+1)}(A+BK^{(i)})
+\sigma^{2}(C+DK^{(i)})'\hat{P}^{(i+1)}(C+DK^{(i)})\no\\
&&+(A+BK^{(i-1)})'P^{(i)}(A+BK^{(i-1)})-(A+BK^{(i)})'P^{(i)}(A+BK^{(i)})\no\\
&&+\sigma^{2}(C+DK^{(i-1)})'P^{(i)}(C+DK^{(i-1)})
-\sigma^{2}(C+DK^{(i)})'P^{(i)}(C+DK^{(i)})\no\\
&&-(K^{(i)})'RK^{(i)}+(K^{(i-1)})'RK^{(i-1)}+(\hat{K}^{(i)})'S+S'\hat{K}^{(i)}-\hat{P}^{(i+1)}\no\\
&=&(A+BK^{(i)})'\hat{P}^{(i+1)}(A+BK^{(i)})
+\sigma^{2}(C+DK^{(i)})'\hat{P}^{(i+1)}(C+DK^{(i)})\no\\
&&+(A'P^{(i)}B+\sigma^{2}C'P^{(i)}D)\hat{K}^{(i)}+(\hat{K}^{(i)})'
(A'P^{(i)}B+\sigma^{2}C'P^{(i)}D)'\no\\
&&+(K^{(i-1)})'(R+B'P^{(i)}B+\sigma^{2}D'P^{(i)}D)K^{(i-1)}
-(K^{(i)})'(R+B'P^{(i)}B+\sigma^{2}D'P^{(i)}D)K^{(i)}\no\\
&&+(\hat{K}^{(i)})'S+S'\hat{K}^{(i)}-\hat{P}^{(i+1)}\no\\
&=&(A+BK^{(i)})'\hat{P}^{(i+1)}(A+BK^{(i)})
+\sigma^{2}(C+DK^{(i)})'\hat{P}^{(i+1)}(C+DK^{(i)})
-\hat{P}^{(i+1)}\no\\
&&-[(R+B'P^{(i)}B+\sigma^{2}D'P^{(i)}D)K^{(i)}]'\hat{K}^{(i)}
-(\hat{K}^{(i)})'
[(R+B'P^{(i)}B+\sigma^{2}D'P^{(i)}D)K^{(i)}]\no\\
&&+(K^{(i-1)})'(R+B'P^{(i)}B+\sigma^{2}D'P^{(i)}D)K^{(i-1)}
-(K^{(i)})'(R+B'P^{(i)}B+\sigma^{2}D'P^{(i)}D)K^{(i)}\no\\
&=&(A+BK^{(i)})'\hat{P}^{(i+1)}(A+BK^{(i)})
+\sigma^{2}(C+DK^{(i)})'\hat{P}^{(i+1)}(C+DK^{(i)})-\hat{P}^{(i+1)}\no\\
&&+(\hat{K}^{(i)})'
(R+B'P^{(i)}B+\sigma^{2}D'P^{(i)}D)\hat{K}^{(i)}.  \label{2.15}
\end{eqnarray}
From Lemma \ref{l1}, \rf{2.15} admits a unique solution $\hat{P}^{(i+1)}\geq0$ due to $(\hat{K}^{(i)})'
(R+B'P^{(i)}B+\sigma^{2}D'P^{(i)}D)\hat{K}^{(i)}\geq0$. Hence, the sequence of $\{P^{(i)}\}^{\infty}_{i=1}$ is monotonically decreasing. The fact that $P^{(i)}>0$, $\{P^{(i)}\}^{\infty}_{i=1}$ is convergent reveals that the limitation can be written as $P$.

Step 2: We show that $P$ is the solution to SARE \rf{2.13}.

From the convergence of $\{P^{(i)}\}^{\infty}_{i=1}$ and ($\ast\ast$), we obtain the convergence of $\{K^{(i)}\}^{\infty}_{i=1}$, i.e., $\lim_{i\rightarrow\infty}K^{(i)}=K$. Hence, from Lemma \ref{l3}, we have
\begin{eqnarray}
&&(A+BK)'P(A+BK)
+\sigma^{2}(C+DK)'P(C+DK)\no\\
&&+K'RK+K'S+S'K+Q-P=0. \label{2.16}
\end{eqnarray}
From Assumption \ref{a2}, we get $K'RK+K'S+S'K+Q=Q-S'R^{-1}S+(K+R^{-1}S)'R(K+R^{-1}S)>0$. Hence, $K$ is a stabilizer of system \rf{2.1} and $P>0$ holds due to Lemma \ref{l1}.

\end{proof}

\section{Online Implementation}

In this section, we illustrate the implementation of Algorithm 1 in detail. Since there are $N:=\frac{n(n+1)}{2}$ independent parameters in matrix $P^{(i+1)}$, we need to observe state along trajectories for at least N intervals $[s_{j},s_{j}+l_{j}]$ with $j=1,2,\ldots,N$ on $[0,\infty)$ to reinforce the target function
\begin{align}
\triangle J^{(i)}(s_{j}, s_{j}+l_{j}, x^{(i)}, K^{(i)}):=\mathbb{E}_{s_{j}}\Big\{\sum^{s_{j}+l_j}_{t=s_{j}}(x^{(i)}_t)'[(K^{(i)})'RK^{(i)}
+(K^{(i)})'S+S'K^{(i)}+Q]x^{(i)}_t\Big\}. \label{3.1}
\end{align}
Denote
\begin{eqnarray*}
\hat{x}_{j}^{(i)}=(x_{s_{j}}^{(i)})'\otimes (x_{s_{j}}^{(i)})'-\mathbb{E}_{s_{j}}[(x^{(i)}_{s_{j}+l_{j}})'\otimes (x^{(i)}_{s_{j}+l_{j}})'],
\end{eqnarray*}
where $x_{s_{j}}^{(i)}$ represents the initial state at the $i$-th iteration, the set of equations \rf{3.1} can be transformed into
\begin{eqnarray}
\begin{bmatrix}\hat{x}_{1}^{(i)}\\\hat{x}_{2}^{(i)}\\\vdots \\ \hat{x}_{N}^{(i)}\end{bmatrix}vec(P^{i+1})=\begin{bmatrix}\triangle J^{(i)}(s_{1}, s_{1}+l_{1})\\
\triangle J^{(i)}(s_{2}, s_{2}+l_{2})\\
\vdots\\
\triangle J^{(i)}(s_{N}, s_{N}+l_{N})\end{bmatrix}.\label{3.2}
\end{eqnarray}
Denote
\begin{eqnarray*}
X^{(i)}=\begin{bmatrix}\hat{x}_{1}^{(i)}\\\hat{x}_{2}^{(i)}\\\vdots \\ \hat{x}_{N}^{(i)}\end{bmatrix},\ \
\mathcal{J}^{(i)}=\begin{bmatrix}\triangle J^{(i)}(s_{1}, s_{1}+l_{1})\\
\triangle J^{(i)}(s_{2}, s_{2}+l_{2})\\
\vdots\\
\triangle J^{(i)}(s_{N}, s_{N}+l_{N})\end{bmatrix},
\end{eqnarray*}
it follows from \rf{3.2} that
\begin{eqnarray}
X^{(i)}vec(P^{i+1})=\mathcal{J}^{(i)}. \label{3.3}
\end{eqnarray}
Furthermore, by using the sampled data at terminal time $s_{j}+l_{j}$, the expectation value of the right hand side of \rf{3.1} can be derived by calculating the mean value based on $L$ sample paths $x^{(i)}_{s_{j}+l_{j}}(k), k=1,2,\ldots, L$, that is,
 \begin{eqnarray*}
 \mathbb{E}_{s_{j}}[(x^{(i)}_{s_{j}+l_{j}})'\otimes (x^{(i)}_{s_{j}+l_{j}})']
 \approx \frac{1}{L}\sum^{L}_{k=1}[(x^{(i)}_{s_{j}+l_{j}}(k))'\otimes (x^{(i)}_{s_{j}+l_{j}}(k))'].
 \end{eqnarray*}
Accordingly, an alternative way to calculate $\triangle J^{(i)}(\cdot)$ is:
\begin{eqnarray*}
\triangle J^{(i)}(s_{j}, s_{j}+l_{j}) =\frac{1}{L}\sum^{L}_{m=1}\Big\{\sum^{s_{j}+l_j}_{t=s_{j}}(x^{(i)}_t(m))'[(K^{(i)})'RK^{(i)}
+(K^{(i)})'S+S'K^{(i)}+Q]x^{(i)}_t(m)\Big\}.
\end{eqnarray*}
From \cite{Murray et al}, there exits a matrix $W\in \mathbb{R}^{n^{2}\times N}$ with $rank(W)=N$ such that $vec(P)=Wvec^{+}(P)$. Thus, \rf{3.3} can be rewritten as
\begin{eqnarray}
(X^{(i)}W)vec^{+}(P^{(i+1)})=\mathcal{J}^{(i)},\label{3.4}
\end{eqnarray}
which leads to
\begin{eqnarray}
vec^{+}(P^{(i+1)})=(X^{(i)}W)^{-1}\mathcal{J}^{(i)}.\label{3.5}
\end{eqnarray}
Finally, we obtain $P^{(i+1)}$ by taking the inverse map of $vec^{+}(\cdot)$.

\section{An RL Algorithm for Partially Unknown System Model}
Due to the fact that it is impractical to require decision makers to obtain full knowledge of system dynamics, we consider the case where matrices $A$ and $B$ are unknown in system model \rf{2.1} and provide an RL method based on least-square estimation.

We first introduce the following notations. Define $\theta^*=\left[\begin{matrix}
	A^*&B^*
\end{matrix}\right]$ as the unknown parameter when designing controllers and $\theta^{(i)}=\left[\begin{matrix}
	A^{(i)}&B^{(i)}
\end{matrix}\right]$, $i=1,2,\cdots$ as the estimates for $\theta$ updated in the $i$-th iteration.  We choose the optimal feedback control based on the current estimate of $\theta^*$ and update the estimate based on the whole trajectories of the state dynamics.

The Lyapunov recursion and optimal feedback control gain associated with $\theta^{(i)}=\left[A^{(i)}\quad B^{(i)}\right]$ are given by:
\begin{align}
	\label{Pi+1}
	0=&	(A^{(i)}+B^{(i)}K^{(i)})'P^{(i+1)}(A^{(i)}+B^{(i)}K^{(i)})
	+\sigma^{2}(C+DK^{(i)})'P^{(i+1)}(C+DK^{(i)})\nonumber\\
	&+(K^{(i)})'RK^{(i)}+(K^{(i)})'S+S'K^{(i)}+Q-P^{(i+1)}£¬\\
	\label{Ki+1}
	\tilde{K}^{(i+1)}=&-(R+(B^{(i)})'P^{(i+1)}B^{(i)}+\sigma^2B'P^{(i+1)}D)^{-1}((B^{(i)})'P^{(i+1)}A^{(i)}+\sigma^2D'P^{(i+1)}C+S).
\end{align}
The key technique is to derive an $\ell_2-$regularized least-squares estimation for $\theta^*$ by collecting trajectories from online experiments. We update the estimates of parameter $\theta$ according to the following DT least-square estimator:
\begin{align}\label{tt1}
	\theta^{(i+1)}=\arg\min\limits_{\theta}\sum\limits_{k=1}^L\sum\limits_{t=s}^{s+l}\Vert x^{(i)}_{t+1}(k)-x_{t}^{(i)}(k)-\theta'Z_t^{(i)}(k)\Vert_2+tr(\theta'\theta),
\end{align}
where
\begin{align}\label{z1}
	Z_t^{(i)}(k)=\left[\begin{matrix}
		x^{(i)}_t(k)\\
		K^{(i)}x^{(i)}_t(k)
	\end{matrix}\right].
\end{align}
Obtaining the derivative with respect to $\theta$ of the right side of (\ref{tt1}) yields:
\begin{align}
	-\sum\limits_{k=1}^L\sum\limits_{t=s}^{s+l}Z_t^{(i)}(k)\left[\left(x^{(i)}_{t+1}(k)-x_{t}^{(i)}(k)\right)'-\left(Z_t^{(i)}(k)\right)'\theta\right]+\theta=0.
\end{align}
By dividing the both sides of equation above by $L$, it follows that $\theta^{(i+1)}$ can be updated by:
\begin{align}\label{theta}
	\theta^{(i+1)}=\left[\frac{1}{L}\sum\limits_{k=1}^L\sum\limits_{t=s}^{s+l}Z_t^{(i)}(k)\left(Z_t^{(i)}(k)\right)'+\frac{1}{L}I\right]^{-1}\left[\frac{1}{L}\sum\limits_{k=1}^L\sum\limits_{t=s}^{s+l}Z_t^{(i)}(k)\left(x^{(i)}_{t+1}(k)-x_{t}^{(i)}(k)\right)'\right]
\end{align}
Now we are ready to present Algorithm 2.

\begin{table}[H]
	
	\centering
	
	\begin{tabular}{l}
		
		\toprule
		
		\textbf{Algorithm 2} \emph{Policy Iteration based on Least-Square Estimation }\\
		
		\midrule
		
		1:\textbf{Initialization:} Select any $\theta^{(0)}=\left[\begin{matrix}A^{(0)}&B^{(0)}\end{matrix}\right]$ leading to stabilizer $K^{(0)}$ for system \rf{2.1}.  \\
		2. For $i=0,1,\cdots$, do: \\
		3. Obtain local state trajectories $x^{(i)}_t$ on $\left[s,s+l+1\right]$ by running system \rf{2.1} with ${K}^{(i)}$.
		\\
		4. Policy Evaluation: Obtain ${P}^{(i+1)}$ by
		\\$
		\xi'{P}^{(i+1)}\xi-\mathbb{E}_s\left[\left(x^{(i)}_{s+l+1}\right)'{P}^{(i+1)}x^{(i)}_{s+l+1}\right]=\mathbb{E}_s\left\{\sum_{t=s}^{s+l}\left(x_t^{(i)}\right)'\left[R+\left({K}^{(i)}\right)'S+S'{K}^{(i)}+R\right]x^{(i)}_t\right\}.
		$\\
		5. Policy Improvement: Update: $K^{(i+1)}$ by (\ref{Ki+1}).\\
		6. Update $\theta^{(i+1)}$ by (\ref{theta}). \\
		7. {until} $\|P^{(i+1)}-P^{(i)}\|<\delta$.\\
		\bottomrule
		%		
		%		3: \textbf{do} \{ \\
		%		
		%		4: Obtain local state trajectories $x^{(i)}$ by running system \rf{2.1} with
		%		$K^{(i)}$ on $[s, s+l+1]$. \\
		%		
		%		5: \textbf{Policy Evaluation} (Reinforcement): Solve $P^{(i+1)}$ from the identity\\
		%		%\end{tabular}
		%		%\end{table}
		%		%%
		%		%\begin{eqnarray}
		%		\(\xi'P^{(i+1)}\xi-\mathbb{E}_{s}[(x^{(i)}_{s+l+1})'P^{(i+1)}x^{(i)}_{s+l+1}]\)\\
		%		\(=\mathbb{E}_{s}\Big\{\sum^{s+l}_{t=s}(x^{(i)}_t)'[Q+(K^{(i)})'S+S'K^{(i)}
		%		+(K^{(i)})'RK^{(i)}]x^{(i)}_t\Big\}\) \quad ($\ast$)\\
		%		%\end{eqnarray}
		%		%%
		%		%\begin{table}[htbp]
		%		%
		%		%   \centering
		%		%
		%		%   \begin{tabular}{l}
			%			%   \\
			%			6: \textbf{Policy Improvement} (Update): Update $K^{(i+1)}$ by the formula\\
			%			\(K^{(i+1)}=-(R+B'P^{(i+1)}B+\sigma^{2}D'P^{(i+1)}D)^{-1}
			%			(B'P^{(i+1)}A+\sigma^{2}D'P^{(i+1)}C+S)\)\quad ($\ast\ast$)
			%			\\
			%			
			%			7: $i\leftarrow i+1$\\
			%			
			%			8:\}\ \textbf{until} $\|P^{(i+1)}-P^{(i)}\|<\delta$.\\
			
			%	\bottomrule
			
		\end{tabular}
		
	\end{table}

\section{Numerical Example}
In this section, the performance of Algorithm is demonstrated by the following numerical examples.
\subsection{Example 1}
We first implement Algorithm $1$ in a discrete-time system where $n=2$, $m=2$ and system matrices $A$, $B$, $C$, $D$ are given by:
$$
A=\left[\begin{matrix}2& 1\\
	0&2\end{matrix}\right],\quad
B=\left[\begin{matrix}
	1&0\\
	-0.5&1
\end{matrix}\right],\quad
C = \left[\begin{matrix}1 & 0\\
	0.5 &1
\end{matrix}\right],\quad
D = \left[\begin{matrix}1& 0.5\\
	0 &1\end{matrix}\right].
$$
The coefficient matrices of cost functional are set as:
$$
Q = \left[
\begin{matrix}10& 5\\
	5 &10
\end{matrix}
\right],\quad
S = \left[\begin{matrix}
	1& 0\\
	0.5 &1\end{matrix}\right],\quad
R=\left[\begin{matrix}
	10&0\\
	0&10
\end{matrix}\right].
$$
For initialization, we choose $K^{(0)}$ as a stabilizer which drives $x_{t}$ to converges to a neighborhood of zero when $t$ approaches to infinity. Here, $K^{(0)}$ is specified as $\left[\begin{matrix}-0.4 &   3.8\\
	-0.5   &-1.4\end{matrix}\right]$. By substituting $K^{(0)}$ into $(\ref{kk})$, we attain the trajectories of $\left[\begin{matrix}
	x_0,x_1,\cdots,x_{50}
\end{matrix}\right]=X=\left[\begin{matrix}
	X_1\\X_2
\end{matrix}\right]$. As shown by Fig.~\ref{fig1.1}, $K^{(0)}$ ensures the state trajectories to be stabilizable. In this case, $P$ contains $N=\displaystyle\frac{n(n+1)}{2}=3$ parameters to be solved. In order to reinforce $\triangle J^{(i)}(s_{j}, s_{j}+l_{j})$, $j=1,2,3,$ we let $x_0$ be $\left[\begin{matrix}
	3\\7
\end{matrix}\right
]$, $\left[\begin{matrix}
	2\\18
\end{matrix}\right
]$,
$\left[\begin{matrix}
	14\\3
\end{matrix}\right
]$
respectively and collect the state information from $t=0,1,\cdots,200$. By recursively implementing Algorithm $1$ for $10$ times, it turns out that
$$
P^*=P^{(10)}=\left[\begin{matrix}
	86.3101&159.5861\\
	159.5861&419.6332
\end{matrix}\right],
$$
and the optimal control is
$$
K^*=K^{(10)}=\left[\begin{matrix}
	-0.6250 &   1.4830\\
	-0.6568  & -1.6745
\end{matrix}\right].$$

Under $K^*$, the optimal trajectories $X^*=\left[\begin{matrix}
	X_1^*\\X_2^*
\end{matrix}\right]$ are depicted by Fig.~\ref{fig1.2}, which reach the stable state faster compared to the trajectories under $K^{(0)}$.
We also compare the calculated solution $P^*$ and the real solution to ARE and derive the error between them, which is given by
$$
\mathcal{R}(P^*)=10^{-12}*\left[\begin{matrix}
	-0.3126&-0.8242\\-0.5969&0.1137
\end{matrix}\right].
$$
To further demonstrate the effectiveness of Algorithm $1$, we assign different values to $R$ and let other system parameters remain unchanged.

When
$
R=\left[\begin{matrix}
	0&0\\0&0
\end{matrix}\right],
$ after $10$ iterations, the solution to ARE and the optimal control calculated by Algorithm $1$ are given by
$$
P^*=\left[\begin{matrix}
	28.9136&49.9422\\
	49.9422&106.2917	
\end{matrix}\right],
$$
$$
K^*=\left[\begin{matrix}
	-0.5103&2.0453\\
	-0.7384&-1.9275
\end{matrix}\right].
$$
Fig.~\ref{fig2.2} shows the optimal trajectories $X^*$. For comparison, we also calculate the standard solution to the SARE and denote it as $P$.
Then the error between $P^*$ and the real solution to ARE is as follows
$$
\mathcal{R}(P^*)=10^{-11}*\left[\begin{matrix}
	0.0728&0.2480\\
	0.2423&0.4761
\end{matrix}\right].
$$

When
$
R=\left[\begin{matrix}
	0&0\\0&-5
\end{matrix}\right],
$ after $10$ iterations, the solution to ARE and the optimal control calculated by Algorithm $1$ are given by
$$
P^*=\left[\begin{matrix}
	28.1154&47.8493\\
	47.8493&100.8027	
\end{matrix}\right],
$$
$$
K^*=\left[\begin{matrix}
	-0.5154&2.0325\\
	-0.7432&-1.9407
\end{matrix}\right].
$$
The optimal trajectories $X^*$ are illustrated in Fig.~\ref{3.2}, which confirms that Algorithm $1$ is still valid for the cost function with negative coefficient matrices. The error between $P^*$ and the real solution to ARE can be obtained as follows
$$
\mathcal{R}(P^*)=10^{-11}*\left[\begin{matrix}
	-0.0288&-0.0689\\
	-0.0803&-0.1194
\end{matrix}\right].
$$
\begin{figure}[H]
	\begin{center}
		\includegraphics[width=8.0cm]{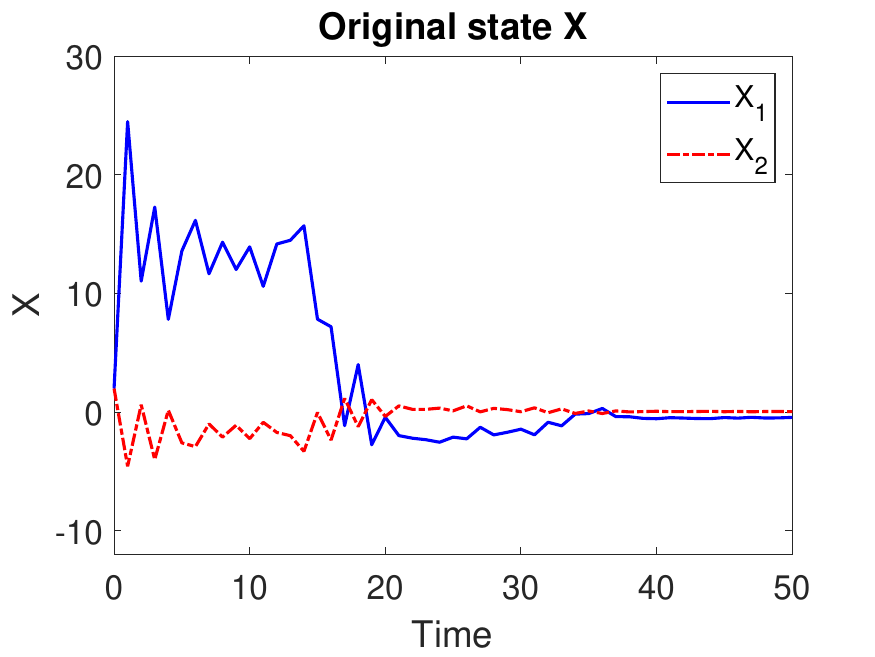}
	\end{center}
	\caption{State trajectories under $K^{(0)}$}\label{fig1.1}
\end{figure}
\begin{figure}[H]
	\begin{center}
		\includegraphics[width=8.0cm]{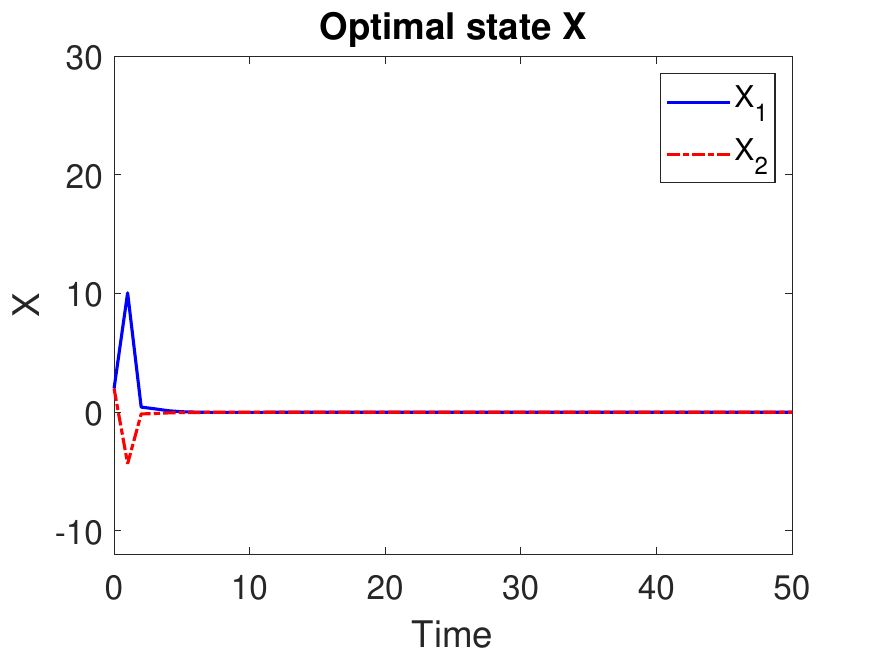}
	\end{center}
	\caption{State trajectories under optimal control with a positive $R$}\label{fig1.2}
\end{figure}
\begin{figure}[H]
	\begin{center}
		\includegraphics[width=8.0cm]{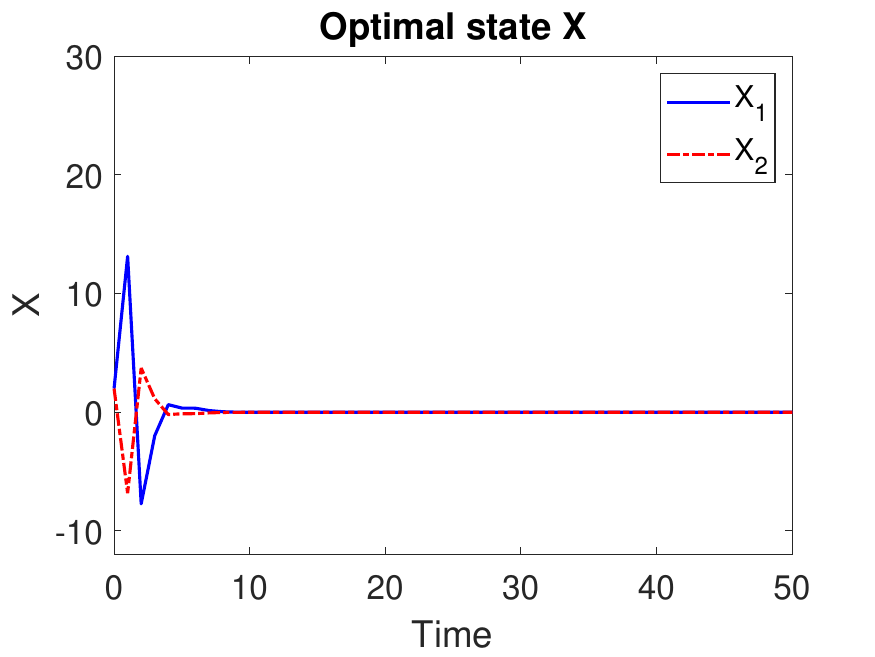}
	\end{center}
	\caption{State trajectories under optimal control with a semi-positive $R$}\label{fig2.2}
\end{figure}
\begin{figure}[H]
	\begin{center}
		\includegraphics[width=8.0cm]{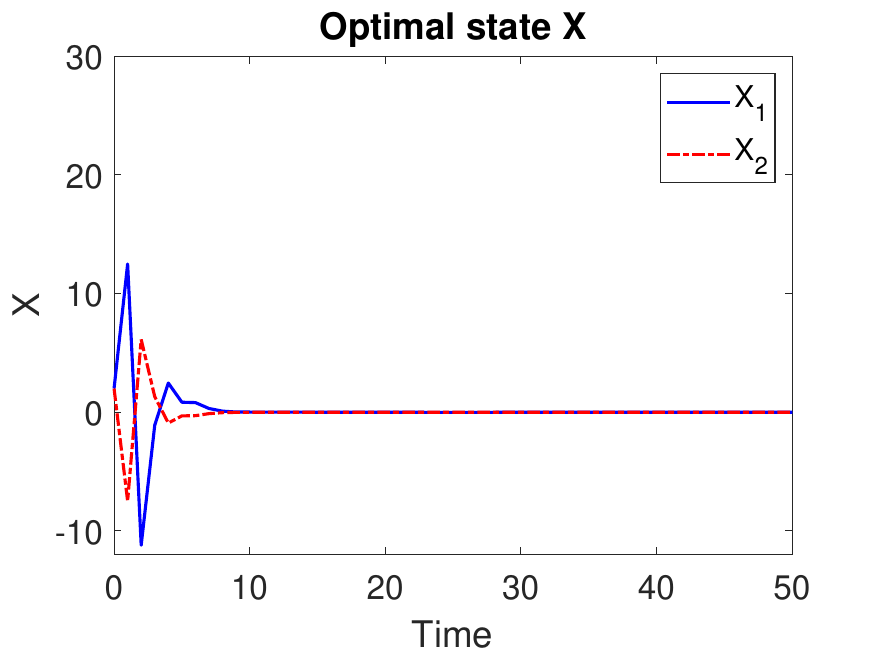}
	\end{center}
	\caption{State trajectories under optimal control with a negative $R$}\label{fig3.2}
\end{figure}

\subsection{Example 2}
We also apply Algorithm 1 to a discrete-time linear system, where the dimension of both state and input control is three and system matrices are specified as:
\begin{align*}
A=\left[\begin{matrix}
	2&1&0\\
	0&2&0\\
	1&0&1
\end{matrix}\right],\quad
B=\left[\begin{matrix}
	1&0&1\\
	-0.5&1&0\\
	0&1&1
\end{matrix}\right],
\quad
C=\left[\begin{matrix}
	1&0&0\\
	0.5&1&1\\
	0&0&1
\end{matrix}\right],\quad
D=\left[\begin{matrix}
	1&0.5&0\\
	0&1&1\\
	0&0&1
\end{matrix}\right],
\end{align*}
\begin{align*}
	Q=\left[\begin{matrix}
		10&5&0\\
		5&10&0\\
		0&0&1
	\end{matrix}\right]\quad
S=\left[\begin{matrix}
	1&0&1\\
	0.5&1&0\\
	0&1&1
\end{matrix}\right],\quad
R=\textbf{diag}\{\begin{matrix}
	10&10&100
\end{matrix}\}.
\end{align*}
In this case, the solution to corresponding SARE contains $\displaystyle\frac{(n+1)n}{2}=\frac{4\times3}{2}=6$ independent entries which need to be evaluated. Hence, for updating $P^{(i)}$, $i=1,2,\cdots$, we collect experiment data from $6$ paths of state trajectories over the time interval of $\left[0,200\right]$. And the initial states are respectively given by:
\begin{align*}\left[\begin{matrix}
	1.69\\1.13\\-0.59\end{matrix}\right], \quad\left[\begin{matrix}
	0.11\\0.75\\-2.10\end{matrix}\right], \quad\left[\begin{matrix}
	0.10\\0.35\\0.58\end{matrix}\right],\quad \left[\begin{matrix}
	-0.08\\0.50\\1.14\end{matrix}\right], \quad\left[\begin{matrix}
	0.10\\-1.91\\0.32\end{matrix}\right], \quad\left[\begin{matrix}
-2.00\\0.30\\0.07\end{matrix}\right].
\end{align*}
Accordingly, the initial feedback control is chosen as
\begin{align*}
	K^{(0)}=\left[\begin{matrix}
 -0.6&   -5.8  &  0.8\\
 -0.3 &  -4.8 &   0.4\\
 -0.7  &  4.8 &  -0.8
\end{matrix}\right].
\end{align*}
Figure 5 reveals that when system $(A,B,C,D)$ is governed by $K^{(0)}$, the $i$-th element of the state trajectories, denoted by $X_i$, for $i=1,2,3$, eventually converge to $0$.

Executing Algorithm 1 for $10$ steps leads to $P^{(10)}$ and $K^{(10)}$, which are the estimates of SARE and optimal control policy, that is
\begin{align*}
	P^{(10)}=& \left[\begin{matrix}44.2609  & 39.3667 &  -3.1271\\
	39.3667  &145.7899  &-30.5478\\
	-3.1271  &-30.5478 &  12.4739\end{matrix}\right],\\	
	K^{(10)} =&\left[\begin{matrix}	
	-1.3493 &  -0.0494&   -0.2034\\
	-0.7136 &  -2.3848  &  0.0886\\
	-0.1755 &  -0.0418&   -0.0554	
\end{matrix}\right],
\end{align*}
which still remain stable after $10$ iterations. Let $P^{*}=P^{(10)}$ and $K^*=K^{(10)}$, the state trajectories driven by $K^*$ are described in Figure 6, where the steady state can be reached after fewer iterations than Figure 5.
For comparison, we also calculate the standard solution to the SARE and denote it as $P$. The estimate error between $P$ and $P^*$ is given by
\begin{align*}
\mathcal{R}(P^*)= P-P^*=10^{-12}\ast\left[\begin{matrix}
0.0284 &   0.0284 &   0.0018\\
-0.1137   &      0   & 0.0213\\
0.0160&    0.0107   &      0
\end{matrix}\right].
\end{align*}
When matrix $R$ is modified as a semi positive-definite matrix, a new $K^*$ can be calculated. Figure 7 displays the state trajectories under $K^*$.
\begin{figure}[H]
	\begin{center}
		\includegraphics[width=8.0cm]{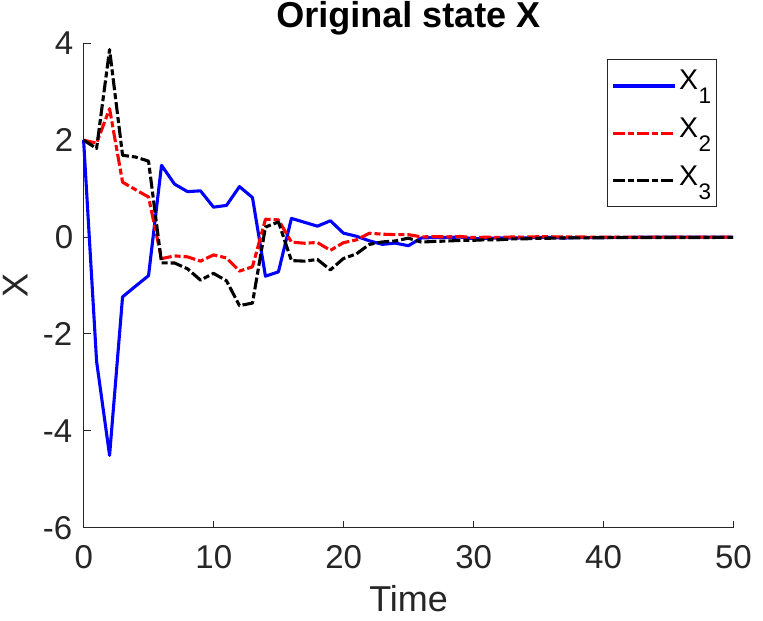}
	\end{center}
	\caption{State trajectories under $K^{(0)}$}\label{fig5}
\end{figure}
\begin{figure}[H]
	\begin{center}
		\includegraphics[width=8.0cm]{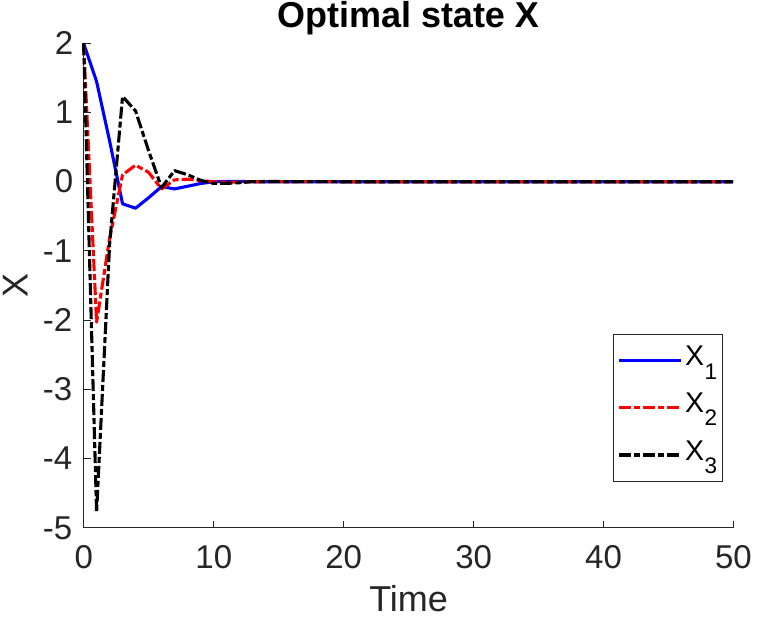}
	\end{center}
	\caption{State trajectories under $K^{*}$ with a positive $R$}\label{fig6}
\end{figure}
\begin{figure}[H]
	\begin{center}
		\includegraphics[width=8.0cm]{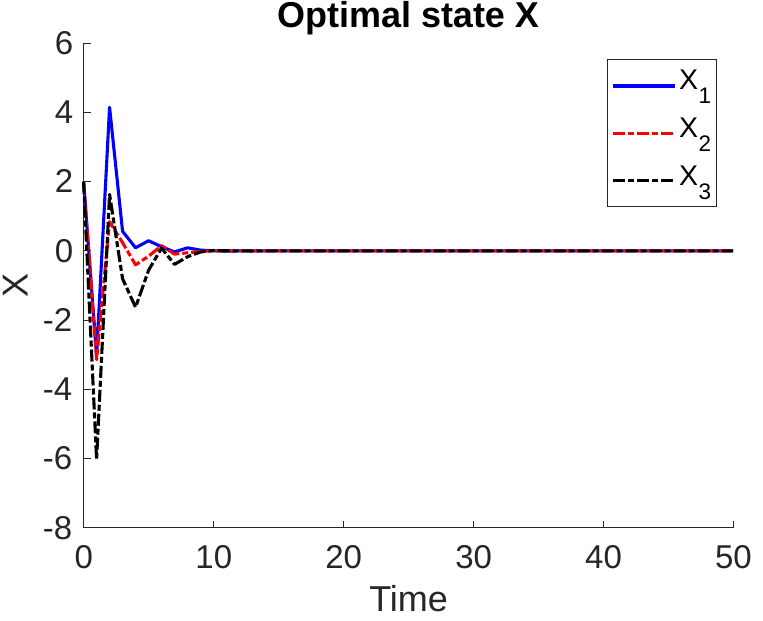}
	\end{center}
	\caption{State trajectories under $K^{*}$ with a semi-positive $R$}\label{fig7}
\end{figure}
\section{Conclusion}
This paper explores the discrete-time stochastic LQ control problem, presenting a novel online RL method. By leveraging state trajectories over short time intervals, the algorithm directly approaches optimal control policies without modeling the system's inner structure. Key contributions include the avoidance of solving the Riccati equation in policy evaluation, using Bellman dynamic programming for reduced computational complexity. The algorithm effectively handles discrete-time stochastic systems with multiplicative noises in both inputs and states. Demonstrating stability without system identification, it proves beneficial for systems with a stabilizable initial controller. Notably, the algorithm's policy evaluation only necessitates trajectory information within arbitrary time intervals, promising efficient and versatile applications.

\bibliographystyle{plain}        % Include this if you use bibtex

\end{document}